\documentclass[12pt]{amsart}
\usepackage{amsmath}
\usepackage{amssymb}
\newtheorem{theorem}{Theorem}
\newtheorem{proposition}[theorem]{Proposition}
\newtheorem{lemma}[theorem]{Lemma}

\newenvironment{proof*}{\vskip 2mm\noindent {}}{\hfill $\Box$ \vskip 2mm}
\newtheorem{example}[theorem]{Example}

\newcommand{\C}{\mathbb C}
\newcommand{\D}{\mathbb D}
\newcommand{\R}{\mathbb R}
\newcommand{\eps}{\varepsilon}
\newcommand{\vphi}{\varphi}
\newcommand{\OO}{\mathcal O}
\def\Re{\operatorname{Re}}
\def\Im{\operatorname{Im}}
\def\dist{\operatorname{dist}}
\def\span{\operatorname{span}}
\begin{document}

\title{Rigid characterizations of pseudoconvex domains}

\author{Nikolai Nikolov}
\address{Institute of Mathematics and Informatics\\ Bulgarian Academy of
Sciences\\1113 Sofia, Bulgaria} \email{nik@math.bas.bg }

\author{Pascal J.~Thomas}
\address{Universit\'e de Toulouse\\ UPS, INSA, UT1, UTM \\
Institut de Math\'e\-ma\-tiques de Toulouse\\
F-31062 Toulouse, France} \email{pthomas@math.univ-toulouse.fr}



\begin{thanks}{This paper was written during the stay of the first-named
author at the Paul Sabatier University, Toulouse (October-November 2010) supported by a
CNRS--BAS programme ``Convention d'\'echanges'' No 23811. He would like to thank Peter Pflug
for helpful discussions about Proposition \ref{12}.}
\end{thanks}

\begin{abstract} We prove that an open set $D$ in $\C^n$ is pseudoconvex if and only if for
any $z\in D$ the largest balanced domain centered at $z$ and contained in $D$ is pseudoconvex,
and consider analogues of that characterization in the linearly convex case.
\end{abstract}

\maketitle

\section{Introduction}
\label{intro}

Geometric convexity of a domain is characterized by its intersection with
real lines, and invariant under real affine maps. Pseudoconvexity is a generalization of
that notion that was designed, among other things, to be invariant under all biholomorphic maps,
and can be characterized by the behavior of analytic disks (Kontinuit\"atsatz). Linear convexity
and $\C$-convexity are intermediate notions that bring into play (respectively)
complex hyperplanes and complex lines, and are invariant under complex affine maps.

In this paper, we exploit the parallels between all those notions, and
highlight the similarities and differences, and the crucial role played
by smoothness of the domains being considered.

\section{Balanced indicators}
\label{balind}

Let $D$ be an open set in $\C^n,$ $z\in D$ and $X\in\Bbb
C^n.$ We say that a domain is \emph{balanced, centered at $a$} if for any $z\in D$,
$\zeta \in \C$ with $|\zeta|\le 1$, then $a+\zeta(z-a) \in D$.

Denote by $d_D(z,X)$ the distance from $z$ to $\partial D$
in the complex direction $X$ (possibly $d_D(z,X)=\infty$):
$$
d_D(z,X)=\sup\{r>0:z+\lambda X\in D\hbox{ if
}|\lambda|<r\}.
$$
Recall that if $-\log d_D(\cdot,X)$ is a plurisubharmonic function for any $X\in\C^n,$
then $D$ is pseudoconvex, and vice versa.

Closely related to this is the largest balanced domain centered at $z$ and contained
in $D$, i.e.
$B_{D,z}=z+I_{D,z},$ where $I_{D,z}$ is the balanced indicatrix of $D$ at $z:$
$$I_{D,z}=\left\{X\in \C^n :z+\lambda X \in D\mbox{\ if\ }|\lambda|\le 1\right\}.$$

Finally, consider the global version of this, the Hartogs-like domain
$$
H_D=\left\{(z,w)\in D\times\C^n:w\in I_{D,z}\right\}.
$$
If $D$ is pseudoconvex,  then $-\log d_D$ is a plurisubharmonic function on $D\times\C^n$
(cf. \cite[Proposition 2.2.21]{Jar-Pfl})\footnote{The authors thank P.~Pflug for pointing
out this fact.}, thus $H_D$ is pseudoconvex.

\subsection{Pseudoconvexity}

The main purpose of this note is
to characterize the pseudoconvexity of an open set $D$ in $\C^n$ in terms of pseudoconvexity
of $B_{D,z},$ $z\in D,$ i.e. in terms of pseudoconvexity in the "vertical" directions of $H_D$.

\begin{theorem}
\label{pcvxhart}
Let $D$ be a proper open set of $\C^n$. Then the following
properties are equivalent:
\begin{enumerate}
\item $D$ is pseudoconvex.
\item $H_D$ is pseudoconvex.
\item $B_{D,z}$ is pseudoconvex, for any $z\in D.$
\end{enumerate}
\end{theorem}

We have already seen that (1) implies (2), and (2) implies (3) is trivial (slice by
the sets $\{z\} \times \mathbb C^n$, for $z\in D$). The remaining implication
is implied by the following.

\begin{proposition}\label{4} Let $D$ be a proper open set of $\C^n$ and let $U$ be a neighborhood
of $\partial D.$ If $I_{D,a}$ is a pseudoconvex domain for any $a\in D\cap U,$ then $D$ is itself pseudoconvex.
\end{proposition}

To prove Proposition \ref{4} (and other propositions below),
we shall use \cite[Theorem 4.1.25]{Hor1}\footnote{The first inequality on p. 242 in the proof must contains
an obvious extra term. Otherwise, it is not true in general; for example, take the domain in $\C^2$ given by
$\Re z<(\Re w)^2.$ But the end result does hold.}, namely

\begin{proposition}\label{hor} If an open set $D$ in $\C^n$ is not pseudoconvex, then there
is a point $a\in\partial D,$ say the origin, and a real-valued quadratic polynomial $q$
such that $q(a)=0,$ $\partial q(a)\neq 0,$
$$\sum_{j,k=1}^n\frac{\partial^2 q}{\partial z_j\partial\overline{z_k}}(a)X_j\overline{X_k}<0$$
for some vector $X\in\C^n$ with $\langle\partial q(a),X\rangle=0,$ and $D$ contains the set $\{q<0\}$ near $a.$

Therefore, after an affine change of coordinates, we may assume $0\in\partial D$ and, near this point, $D$ contains
the set $$\{z\in\C^n:0>\Re z_1+(\Im z_1)^2+|z_2|^2+\dots+|z_{n-1}|^2+c(\Im z_n)^2-(\Re z_n)^2\},$$
where $c<1.$
\end{proposition}

\begin{proof*}{\it Proof of Proposition \ref{4}.}
Assume that $D$ is not pseudoconvex. By Proposition \ref{hor}, we may suppose that
$$D\supset E=\{(z,w)\in\D^2_\eps:\rho(z,w)<0\},$$
where $\rho(z,w)=\Re z+(\Im z)^2-(\Re w)^2+c(\Im w)^2$ and $c<1$
($\D^2_\eps$ is the bidisc with center $0$ and radius $\eps>0$).

For $\delta>0$ and $X\in\C^2,$ let $z_\delta=(-\delta,0)$ and $r_\delta(X)=d_E(z_\delta,X)$.
We write, for $\eta \in \C$, $X_\eta=(\delta,\eta)$.

\begin{lemma}\label{5} For any small $\delta>0$ and $\delta\ge s\ge3(1-c)^{-1/2} \delta^{3/2},$
$$\int_0^{2\pi} \frac1{r_\delta (X_{se^{i \theta}})}\frac{d\theta}{2\pi} <1.$$
\end{lemma}

Assuming Lemma \ref{5}, set $\C^n\ni\tilde z_\delta=(-\delta,0,\dots,0)$ and $\C^n\ni\tilde X_\eta=(\delta,\eta,0,\dots,0).$
Since $r_\delta\le d_{D,\tilde z_\delta}$ and $r_\delta(X_0)=d_D(\tilde z_\delta,X_0)=1$ for $\delta$ small enough,
it follows that $h_{\delta}=1/d_D(\tilde z_\delta,\cdot)$ is not a plurisubharmonic function, which implies that the
balanced domain $I_{D,\tilde z_\delta}$ (with Minkowski function $h_\delta$) is not a pseudoconvex domain
(cf. \cite[Proposition 2.2.22 (a)]{Jar-Pfl}). This contradiction proves Proposition \ref{4}.
\end{proof*}

Lemma \ref{5} will be proved at the end of this section.

\subsection{Linear convexity}

It is interesting to note that a similar statement holds for linear convexity.
Recall that (cf. \cite{Hor1}) a open set $D$ in $\C^n$ is called
weakly linearly convex (resp. linearly convex) if for any $a\in\partial D$
(resp. $a\in\C^n\setminus D$) there exists a complex hyperplane $T_a$ through $a$ which does
not intersect D (such a set is necessarily pseudoconvex). We call $T_a$ a supporting complex hyperplane.
A domain $D$ in $\C^n$ is said to be $\C$-convex
is any nonempty intersection of $D$ with a complex line is connected and simply connected.
All three notions coincide for $C^1$-smooth open sets.

Note that an open balanced set is weakly linearly convex if and only if
it is convex. It is also known that if $D$ is weakly linearly convex, then
$B_{D,z}$ is a convex domain for any $z\in D$ (i.e. the Minkowski function $1/d_D(z,\cdot)$ of $I_{D,z}$ is convex).

\begin{theorem}
\label{lincvxhart}
Consider the following three properties:
\begin{enumerate}
\item $D$ is weakly linearly convex (resp. linearly convex).
\item $H_D$ is weakly linearly convex (resp. linearly convex).
\item $B_{D,z}$ is (weakly linearly) convex, for any $z\in D.$
\end{enumerate}
Then (1) and (2) are equivalent, and imply (3).
If $D$ is a $C^{1,1}$-smooth bounded domain, then (3) implies (1).
\end{theorem}

The last statement follows from \cite{Nik-Tho}. Note that in this
case, the domain $D$ is in fact $\C$-convex.  The domain $H_D$, however,
does not share the smoothness of $D$, and may fail to be $\C$-convex.

\begin{example}\label{10}If $D=\{z\in\C:|z-1|<2\ or\ |z+1|<2\},$ then $H_D$ is not $\C$-convex.
\end{example}

\begin{proof} The set $H_D\cap(\C\times\{\sqrt3\})$ is not connected.
\end{proof}

\begin{proof*}{\it Proof of Theorem \ref{lincvxhart}.}
Since $D=(\C^n \times \{0\})\cap H$, (2) implies (1). To prove the converse,
may assume that
$D\neq\C^n.$ Let $D$ be weakly linearly convex (resp. linearly convex) and
$(a,b)\in\partial D$ (resp. $(a,b)\in\C^n\setminus D$). It follows that $c=a+\lambda_0 b\in\partial D$
(resp. $c\in\C^n\setminus D$) for some $\lambda_0\in\C$ with $|\lambda_0|\le 1.$
There exists a supporting complex hyperplane for $D$ at $c,$ say $T_c=\{z\in\C^n:L(z)=0\},$ where $L:\C^n\to\C$
is an affine map. Then $T_{a,b}=\{(z,w)\in\C^n\times\C^n:L(z+\lambda_0w)=0\}$ is a supporting complex hyperplane
for $H_D$ at $(a,b).$
\end{proof*}

If we turn to the third, and more usual notion of convexity, it is clear that a domain $D$ in $\R^n$
is convex if and only if $H_D$ is convex in $\R^n\times\R^n.$

\subsection{Proof of Lemma \ref{5}.}

Set $\zeta = r e^{i\alpha}$, $-\pi<\alpha\le\pi$,
and $C_1= 3 \sqrt{\delta}$. We first estimate
$\rho(z_\delta + \zeta X_{se^{i \theta}})$ when $|\alpha|\ge C_1$:
\begin{multline*}
\rho(z_\delta + \zeta X_{se^{i \theta}}) \le
\delta \left( -1 +r \cos \alpha + \delta r^2 \sin^2 \alpha \right) + c r^2 s^2
\\
\le \delta (-1 +r \cos C_1) + 2 r^2 \delta^2
\le \delta (-1 +r (\cos C_1+ 2 \delta)).
\end{multline*}
For any small $\delta,$ $\cos C_1 \le 1- C_1^2/3 = 1-3\delta$, so if
$r<(1-\delta)^{-1}$, then $\rho(z_\delta + \zeta X_{e^{i \theta}})<0$.

Now we estimate
$\rho(z_\delta + \zeta X_{se^{i \theta}})$ when $|\alpha|\le C_1$. Notice that
$$
\rho(z_\delta + \zeta X_{se^{i \theta}})=
\delta \left( -1 +r \cos \alpha + \delta r^2 \sin^2 \alpha \right)
+r^2 s^2 \left( (1+c) \sin^2 (\alpha+\theta) -1 \right).
$$
It is easy to check that
$$
\sin^2 (\alpha+\theta) \le \sin^2 \theta + \sin|\alpha| \le \sin^2 \theta + C_1,
$$
so
\begin{multline*}
\rho(z_\delta + \zeta X_{se^{i \theta}}) \le
\delta \left( -1 +r +  \delta r^2 C_1^2 \right)
+r^2 s^2 \left( (1+c) \sin^2 \theta+(1+c)C_1 -1 \right)
\\
= \delta \left( -1 +r +r^2 A\right),
\end{multline*}
where $A=A(\theta):= \delta C_1^2 +\frac{s^2}\delta \left( (1+c) \sin^2\theta+(1+c)C_1 -1 \right)$.
Notice that $-1\le-\delta<A$  for $s\le \delta\le 1.$

Suppose that $1/r>1+A$, then
$$
\frac1\delta \rho(z_\delta + \zeta X_{e^{i \theta}})
<-\frac{A^2}{(1+A)^2} \le 0.
$$

Putting together both estimates, for $\delta$ small enough,
$$
r_\delta (X_{se^{i \theta}}) > \min\left( \frac1{1-\delta}, \frac1{1+A}\right)
= \frac1{1+A}.
$$
Therefore
\begin{multline*}
\int_0^{2\pi} \frac1{r_\delta(X_{se^{i \theta}})}\frac{d\theta}{2\pi} <
\int_0^{2\pi} (1+A(\theta))\frac{d\theta}{2\pi}  \\
= 1 + \delta C_1^2 +\frac{s^2}\delta \left( (1+c) \frac12 +(1+c)C_1 -1 \right)
\le 1+ 3\delta^2 - \frac{s^2(1-c)}{3\delta}\le 1
\end{multline*}
for any small $\delta.$\qed

\section{Defining functions}
\label{deffcn}

\subsection{Convexity}

We point out that the proof that the convexity of $B_{D,z}$
implies linear convexity for $\mathcal C^{1,1}$ domains
\cite[Proposition 1 \& introduction]{Nik-Tho}
is based on the following which can be easily deduced from \cite{Hor2}.
Let $s_D$ stand for the signed distance to $\partial D$.

\begin{proposition}\label{2} If $D$ is a $C^{1,1}$-smooth bounded domain in $\C^n$ and
$$
\liminf_{T^\C_a\ni z\to a}\frac{s_D(z)}{|z-a|^2}\ge 0\ \footnote{When $\partial D$ is twice
differentiable at $a,$ this limit is equal to the minimal eigenvalue of $2\mbox{Hess}_{s_D}(a)|_{T^\C_a}$.}
$$
for $a\in\partial D$ almost everywhere, then $D$ is linearly convex.
\end{proposition}

Proposition \ref{2} has an obvious convex analog.

\begin{proposition}\label{3} A proper domain $D$ in $\R^n$ is convex if and only
if for any $a\in\partial D$ there exists a (real) hyperplane $S_a$ through $a$ such that
$$\liminf_{S_a\ni x\to a}\frac{s_D(x)}{|x-a|^2}\ge 0.$$
\end{proposition}

If $D$ is convex, then obviously $S_a$ is a (real) supporting hyperplane.

\begin{proof} The necessarity is clear by taking supporting hyperplanes.

Assume now that $D$ is not convex. By \cite[Theorem 2.1.27]{Hor1}, one may find $a\in\partial D$
and a smooth domain $G\subset D$ such that $a\in\partial G$ and  $2\mbox{Hess}_{s_G}(a)_{|T^\R_a}$
has an eigenvalue $\lambda<0.$ Since
$s_D\le s_G,$ it follows that
$$\liminf_{T^\R_a\ni x\to a}\frac{s_D(x)}{|x-a|^2}\le\lambda$$
and
$$\liminf_{t\to 0}\frac{s_D(a+tX)}{t^2}=-\infty,\quad a+X\not\in T^\R_a$$
which implies the sufficient part.
\end{proof}

Clearly, the relationship between a domain and its defining function
is not symmetric, as convexity of one sublevel set (or indeed, of all of
them) cannot imply convexity of the function: simply compose
by a monotone increasing function from the real line to itself.  Given a convex domain,
the question arises of how to choose a convex defining function, and
of how much choice one may have.

By \cite[Proposition]{Her-McN},
a smooth bounded domain $D$ is convex if and only if $-\log s_D$ is convex near $\partial D.$
Thanks to \cite[Theorem 2.1.27]{Hor1}, this result can be easily generalized.

\begin{proposition}\label{6} Let $f:\R^+\to\R$ be a nonconstant decreasing and convex function.
Let $U$ be a neighborhood of the boundary of a proper domain $D$ in $\R^n.$ Then $D$ is convex
if and only if $f\circ s_D$ is a convex function on $D\cap U.$
\end{proposition}

In particular, if any of the defining functions given above
is convex on a neighborhood of $\partial D$, then all the others are.

\begin{proof} If $D$ is convex, then $s_D$ is concave and thus
$g=f\circ s_D$ is convex on $D.$

To prove the converse, assume that $g$ is convex on $D\cap U$ but $D$ is not convex.
By \cite[Theorem 2.1.27]{Hor1}
(see the proof of Proposition \ref{3}), we may find a segment $[a,b]\in D\cap U$ such that
$s_D(m)<s_D(x)$ for any $x\in[a,b]\setminus\{m\},$ where $m=\frac{a+b}{2}.$ On the other hand,
it follows that $f$ is
strictly decreasing and then $g(a)+g(b)<2g(m),$ a contradiction.
\end{proof}

Note that it is necessary to require that
the function $f$ be decreasing and convex as the following
example shows.

\begin{example}\label{7} Let $D=\R^+\times\R^+$ and let $f:\R^+\to\R$ be a nonconstant function
such that $f\circ s_D$ is a convex function on $D.$ Then $f$ is decreasing and convex.
\end{example}

\begin{proof} Let $g=f\circ s_D.$ Since $g(t,t)=f(t)$ for $t>0,$ it follows that $f$ is convex.
On the other hand, $2f(t)=2g(t,t)\le g(p,t)+g(2t-p,t)=f(p)+f(t),$ i.e. $f(t)\le f(p)$ for $0<p\le t.$
\end{proof}

\subsection{Pseudoconvexity}

The pseudoconvex analog of Proposition \ref{6} is the following.

\begin{proposition}\label{8} Let $f:\R\to\R$ be a nonconstant increasing and convex function. Let $U$ be a
neighborhood of the boundary of a proper domain $D$ in $\C^n.$ Then $D$ is pseudoconvex if and only if
$f\circ q_D$ is a plurisubharmonic function on $D\cap U,$ where $q_D=-\log s_D.$
\end{proposition}

\begin{proof} If $D$ is pseudoconvex, then $q_D$ is plurisubharmonic and thus $g=f\circ s_D$ is
plurisubharmonic on $D.$

To prove the converse, assume that $g$ is plurisubharmonic on $D\cap U$ but $D$ is not pseudoconvex.
Note that there is $m\in\R$ such that $f(x)$ is strictly increasing for $x>m.$ We may assume that
$-\log s_D>m$ on $D\cap U.$ Using Proposition \ref{hor},
we may easily find a quadratic polynomial map
$p\in\OO(\D,D\cap U)$ ($\D$ is the unit disc) such that $s_D(p(0))<s_D(p(\zeta))$ for any $\zeta\in\D_\ast$
(strong Kontinuit\"atsatz). Then $g(p(0))>g(p(\zeta))$ which contradicts to the maximum principle for the subharmonic
function $g\circ p.$
\end{proof}

Note that there is a smooth bounded pseudoconvex domain in $\C^2$ having no defining function which is
plurisubharmonic on a two-sided neighborhood of the boundary. We do not know under which general conditions on $f$ the
plurisubharmonicity of $f\circ s_D$ is equivalent to the pseudoconvexity of $D.$

We also point out that the proofs of Propositions \ref{3} and \ref{8}
imply a similar result to Propositions \ref{2} and \ref{3} in
the pseudoconvex case:

\begin{proposition}\label{9} If $D$ is a proper open set in $\C^n$ and
for any $a\in\partial D$ there exists a complex hyperplane $S_a$ through $a$ such that
$$
\liminf_{S_a\ni z\to a}\frac{s_D(z)+s_D(a+J(z-a))}{|z-a|^2}\ge 0,
$$
where $J$ is the standard complex structure, then $D$ is pseudoconvex.
\end{proposition}

The converse is also true if $D$ is a $C^2$-smooth open set.
We do not know if the smoothness can be weakened.

\section{Slicing}
\label{slice}

It is known that an open set $D$ in $\C^n$ ($n\ge 3$) is
pseudoconvex if and only if any 2-dimensional slice of $D$ is
pseudoconvex \cite{Hit} (see also \cite{Jac}). Following the idea
in \cite{Jac}, we would like to restrict the family of slices
that has to be used in order to detect pseudoconvexity, namely
we would like to consider the family of complex planes passing
through a common point $a \in \C^n$. As the next results show, it will be
enough generically.  Given $D$ be a open non-pseudoconvex set in $\C^n$,
call $a$ \emph{exceptional} with respect to $D$ if for any $2$-dimensional complex plane
$P \ni a$, $P\cap D$ is pseudoconvex.  The next proposition shows that
the set of exceptional points has to be contained in a complex hyperplane.

\begin{proposition}\label{11} Let $D$ be a open non-pseudoconvex set in $\C^n$ ($n\ge 3$).
Let $S$ be the union of all $2$-dimensional complex planes with non-empty and
non-pseudoconvex intersections with $D,$ so the set of
exceptional points is $\C^n\setminus S$. Then there exists a complex hyperplane $T$
such that $\C^n\setminus S\subset T.$
\end{proposition}

\begin{proof} By Proposition \ref{hor},
we may suppose that $0\in\partial D$ and $D\supset G\cap\D^n(0,\eps)$ for
some $\eps>0,$ where
$$G=\{z\in\C^n:0>r(z)=\Re z_1+||z||^2-c(\Re z_n)^2\}, \quad c>2.$$

Choose a point $a\in\C^n$ which does not belong to the complex tangent
hyperplane to $\partial G$ at $0$, i.e.
with non-zero first coordinate. It is enough to show that if
$L=\span(\overrightarrow{a},\overrightarrow{e_n}),$ then $D'=D\cap
L$ is not pseudoconvex.
For this, note that since $\{z_1=0\}\cap L$ is a transverse intersection,
$G\cap L$ is a smooth domain near $0$ and since $\C e_n \subset \{z_1=0\}\cap L$,
the Levi form of $r|_L$ has a negative
eigenvalue at $0.$ Set $G'=G\cap L\cap\D^n(0,\eps).$ Then there is
a quadratic polynomial map
$\vphi\in\OO(\D,G')\subset\OO(\D,D')$ with
$$||\vphi(0)||<\dist(\vphi(\zeta),\partial G')\le
\dist(\vphi(\zeta,\partial D'),\quad\zeta\in\D_\ast$$ (we have already used this argument
in the proof of Proposition \ref{8}) which shows
that $D'$ is not pseudoconvex.
\end{proof}

If $D$ is $C^2$-smooth, the set of exceptional points with respect to $D$
has to be smaller (compare with Example \ref{14}(i)).

\begin{proposition}\label{12} Let $D$ and $S$ be as in Proposition \ref{11}.
If $D$ is $C^2$-smooth and non-pseudoconvex near some boundary point, then there
exists a complex plane $T$ of codimension $3$ such that $\C^n\setminus S\subset T.$
\end{proposition}

\begin{proof} Note that the respective boundary point, say $0,$
satisfies the conclusions of Proposition \ref{hor}.
We have the same for any point $a\in\partial D$ near $0.$ Assume that
$u,v\not\in S.$ Let $r$ be a $C^2$-smooth
defining function for $D$ near $0.$ By the proof of Proposition
\ref{11}, the complex tangent hyperplane
to $\partial D$ at any such $a$ contains $u$ and $v.$ Then it is easy
to see that the derivative of $r$ in direction
$u-v$ vanishes at any such $a.$ So, if $\C^n\setminus S$ does not lie
in a complex plane of codimension $3$, we may
assume that near $0,$ $r$ depends only on  two coordinates, $z_1$ and $z_n$,
and so we can take (as in the proof above)
$$
G=\{z\in\C^n: 0>r_G(z):=\Re z_1+(\Im z_1)^2+|z_n|^2-c(\Re z_n)^2\}, \quad c>2.
$$

We already know that for any exceptional point $b$, $b_1=0$. Suppose
that $b_n\neq 0$. Consider the complex plane $P$ through $0$ spanned
by $b$ and $e_1$. Then
the complex tangent space to $P\cap D$ at $0$ is $\C e_n$,
$$
r_G (\zeta e_1 + \xi b) = r_G (\zeta_1 e_1 + \xi b_n e_n),
$$
so that the intersection with that plane is not pseudoconvex at $0$.
Therefore the set of exceptional points is contained in $\{z_1=z_n=0\}$.

Let $D'=D\cap\{z_2=\dots=z_{n-1}=0\}.$ The complex tangent space to
any point in $\partial D$ near $0$ must pass through an exceptional point,
so the complex tangent line to any point in $\partial D'$ near $0$
passes through $0$. In particular, the same holds
for the real tangent hyperplanes. 

Taking a $C^1$-smooth defining function of $D'$ near $0$ of the form $x_1-\rho(u),$ where $u=(y_1,x_n,y_n),$
we get the Euler differential equation
$$
\rho(\tilde u)=\frac{\partial\rho(\tilde u)}{\partial y_1}\widetilde{y_1}+
\frac{\partial\rho(\tilde u)}{\partial x_n}\widetilde{x_n}+\frac{\partial\rho(\tilde u)}{\partial y_n}\widetilde{y_n}.
$$
Hence $\rho$ is a homogeneous function of order $1$ and the $C^1$-smoothness near $0$ implies that $\rho$ is linear.
It follows that $\partial D'$ is a hyperplane near $0$ and so it is pseudoconvex there, which is a contradiction.
\end{proof}

The following example shows that there can be an exceptional point even in the 3-dimensional case, when
the boundary is smooth except one point.

\begin{example}\label{13} There exists a an unbounded domain in $\C^3$ with real-analytic boundary except
one point, the origin, which has exactly one exceptional point, namely the origin.
\end{example}

\begin{proof}
Let $\Omega= \{|z_3|^2<|z_1|^2+|z_2|^2<4|z_3|^2\}$ and
$\rho(z)=|z_3|^2-|z_1|^2-|z_2|^2.$ For a point
$0\neq z^0=(z_1^0, z_2^0,z_3^0)$ with $\rho(z^0)=0,$
the restriction of $\rho$ to the horizontal complex line inside the
complex tangent
hyperplane is
$$
\rho \left(  z_1^0-\lambda \bar z_2^0, z_2^0 +\lambda \bar z_1^0,z_3^0\right)
= -|\lambda|^2 (|z_1|^2 +|z_2|^2 ).
$$
Using homogeneity,
it is easy to check that the Levi form of $\rho$ is semidefinite
negative with one strictly negative eigenvalue (in particular,
$\rho \equiv 0$ along the line through the origin and $(z_1^0, z_2^0,z_3^0)$).

Now let $P=\{ \alpha_1 z_1 + \alpha_2 z_2 + \alpha_3 z_3 =0\}$. If
$\alpha_3\neq 0$, since $\Omega$ is invariant under  rotations in the
$(z_1,z_2)$-plane, we may assume
$P=\{ z_3= \alpha z_1\},$ where $\alpha\ge 0.$  Set $D_\alpha=\{z\in
\C^3:|z_2|^2<(4\alpha^2-1)|z_1|^2\}$ for $\alpha>1/2$,
and $G_\alpha=\{z\in \C^3:(\alpha^2-1)|z_1|^2<|z_2|^2\}$ for
$\alpha>1.$ Note that both domains are pseudoconvex.
Then $\Omega\cap P=\emptyset$ if $\alpha\le 1/2,$
$\Omega\cap P=D_\alpha \cap P$ if $1/2<\alpha < 1$,
$\Omega\cap P=D_\alpha\cap P\cap(\C\times\C_\ast\times\C)$ if $\alpha = 1$ and
  and
$\Omega\cap P=D_\alpha\cap G_\alpha\cap P$ if $\alpha>1$; these
intersections are
pseudoconvex.

If $\alpha_3=0$, using a rotation again, we may assume $P=\{z_2=0\}$. Then
$\Omega\cap P=\{|z_3|^2<|z_1|^2<4|z_3|^2\}$ which is pseudoconvex.

So, the origin is an exceptional point.

On the other hand, it follows by the proof of Proposition \ref{11}
that an exceptional point belongs to the
complex tangent hyperplane to $\partial D$ at any $z^0$ as above. This
implies that the origin is the only
exceptional point.
\end{proof}

In the 3-dimensional case we may have more than one exceptional point.

\begin{example}\label{14} Let $a\in\C^3,$ $G$ be a pseudoconvex set in $\C^3,$
and let $l_1,l_2$ be distinct complex lines in $\C^3$ that intersect $G.$ Then:

(i) any intersection of $G\setminus l_1$ with a 2-dimensional complex plane through $a$
is pseudoconvex if and only if $a\in l_1\setminus G.$

(ii) any intersection of $G\setminus(l_1\cup l_2)$ with a 2-dimensional complex plane through $a$
is pseudoconvex if and only if $G\not\ni a=l_1\cap l_2.$
\end{example}

\begin{proof} (i) Let $P$ be a 2-dimensional complex hyperplane through $a.$
Let first $a\in l_1\setminus G.$ If $l_1\not\subset P,$ then $G_1:=(G\setminus l_1)\cap P=G\cap P=:G_P$
is pseudoconvex. Otherwise, $G_1$ is pseudoconvex as the intersection of the pseudoconvex sets
$G_P$ and $P\setminus l_1.$

Let now $a\not\in l_1\setminus G.$ If $a\in G\cap l_1$ and $P$ contains no $l_1,$ then
$G_1=G_P\setminus\{a\}$ is not pseudoconvex. Otherwise, $a\not\in l_1$ and if $P$ intersect $l_1$
at $b\in G,$ then $G_1=G_P\setminus\{b\}$ is not pseudoconvex.

(ii) The proof is similar to that of (i) and we skip it.
\end{proof}

Using Proposition \ref{2}, similar arguments as in the proof of Proposition \ref{11}
implies that $a$ is a point in $C^2$-smooth domain $D$
such that any non-empty intersection of $D$ with a 2-dimensional
complex plane through $a$ is weakly linearly convex, then
$D$ is $\C$-convex.

The following example shows that we have no such  phenomenon in general.

\begin{example}\label{15}
Let
$$D=\{z\in\C^3:|z|<\sqrt2\max\{|z_1|,|z_2|,|z_3|\}\}.$$
Then $D$ is a union of three disjoint linearly convex domains and $D$
has a non-empty linearly convex intersection with any complex plane through $0$
(in particular, $D$ is pseudoconvex and not weakly linearly convex).
\end{example}

\begin{proof}
Letting $D_j= \{ z\in\C^3: |z_j|^2 > \sum_{1\le k \le 3, k\neq j} |z_k|^2\}$, we see that $D=\cup_{j=1}^3 D_j$. Clearly the $D_j$ are pairwise disjoint, and obtained one from the other by unitary transformations (permutations of coordinates).

$\C^3 \setminus D_3 $ is the union of the complex planes of the form
$\{z_3= \alpha_1 z_1 + \alpha_2 z_2\}$ for $|\alpha_1|^2 + |\alpha_2|^2 \le 1\}$.  So all $D_j$ are linearly convex.

For any complex plane $P$, $D\cap P$ is a union of punctured complex lines through $0$, so (as in the proof of Example \ref{13}), it is pseudoconvex
if and only if it is smaller than $P\setminus \{0\}$. But $P\setminus \{0\}$
is connected, so if $P\setminus \{0\} \subset D$, then there exists a $j$ such that $P\setminus \{0\} \subset D_j$. Since $D_j$ is linearly convex, it must be pseudoconvex, so by Hartog's phenomenon, it would contain $0$, a contradiction.

To finish the proof and show that $D$ is not linearly convex, we will show that its complement contains no complex plane $P$. By permuting coordinates, we may assume that $P= \{z_3= \alpha_1 z_1 + \alpha_2 z_2\}$, with
$|\alpha_1| \le 1$, $|\alpha_2| \le 1$.  If we suppose that one of those
inequalities is strict, say $|\alpha_1| < 1$, then the points in $P$
such that $z_2=0$ verify $|z_1|^2 > |\alpha_1 z_1|^2  = |z_3|^2 + |z_2|^2$
and $P\cap D_1 \neq \emptyset$. If $|\alpha_1| =|\alpha_2| = 1$,
there are points $(z_1,z_2,z_3)\in P$ such that $ |z_3|=  |z_1|+ |z_2|$
and $z_1z_2 \neq 0$,
so $ |z_3|^2 = |z_1|^2 + |z_2|^2 + |z_1| |z_2| > |z_1|^2 + |z_2|^2$, thus
$P\cap D_3 \neq \emptyset$.
\end{proof}

In spite of Example \ref{15}, one may also conjecture the following:
\smallskip

If $D$ is an open set in $\C^n$ such that any non-empty
intersection with 2-dimensional complex plane is (weakly)
linearly convex, then $D$ is (weakly) linearly convex.

\end{document}